\documentclass[12pt,reqno]{article}

\usepackage[usenames]{color}
\usepackage{amssymb}
\usepackage{graphicx}
\usepackage{amscd}

\usepackage[colorlinks=true,
linkcolor=webgreen,
filecolor=webbrown,
citecolor=webgreen]{hyperref}

\definecolor{webgreen}{rgb}{0,.5,0}
\definecolor{webbrown}{rgb}{.6,0,0}

\usepackage{color}
\usepackage{fullpage}
\usepackage{float}

\usepackage{graphics,amsmath,amssymb}
\usepackage{amsthm}
\usepackage{amsfonts}
\usepackage{latexsym}
\usepackage{epsf}
\usepackage{url}

\def\modd#1 #2{#1\ \mbox{\rm (mod}\ #2\mbox{\rm )}}
\def\red#1{\textcolor{red}{#1}}
\def\divides{{ \, | \,}}
\def\tothe{\uparrow}
\def\Zee{{\mathbb{Z}}}

\def\Que{{\mathbb{Q}}}
\DeclareMathOperator{\rad}{rad}

\newcommand{\seqnum}[1]{\href{http://oeis.org/#1}{\underline{#1}}}

\theoremstyle{plain}
\newtheorem{theorem}{Theorem}

\newtheorem{lemma}[theorem]{Lemma}
\newtheorem{proposition}[theorem]{Proposition}

\theoremstyle{definition}

\newtheorem{example}[theorem]{Example}
\newtheorem{conjecture}[theorem]{Conjecture}

\theoremstyle{remark}
\newtheorem{remark}[theorem]{Remark}

\begin{document}

\title{The Generalized Nagell-Ljunggren Problem: \\  
Powers with Repetitive Representations}

\author{Andrew Bridy \\
Department of Mathematics \\
Texas A{\&}M University \\
Mailstop 3368 \\
College Station, TX 77843-3368  \\
USA \\
\href{mailto:andrewbridy@math.tamu.edu}{\tt andrewbridy@math.tamu.edu} \\
\ \\
\and
Robert J. Lemke Oliver\\
Department of Mathematics\\
Tufts University\\
Medford, MA  02155\\ 
USA\\
\href{mailto:robert.lemke_oliver@tufts.edu}{\tt robert.lemke\_oliver@tufts.edu} \\ 
\ \\
\and
Arlo Shallit \\
Toronto, Ontario \\
\and 
Jeffrey Shallit \\
School of Computer Science \\
University of Waterloo \\
Waterloo, ON  N2L 3G1 \\
Canada \\
\href{mailto:shallit@cs.uwaterloo.ca}{\tt shallit@cs.uwaterloo.ca}
}

\maketitle

\begin{abstract}
We consider a natural generalization of the Nagell-Ljunggren equation
to the case where the $q$th power of an integer $y$, for $q \geq 2$,
has a base-$b$ representation that
consists of a length-$\ell$
block of digits repeated $n$ times, where $n \geq 2$.  
Assuming the $abc$ conjecture of Masser and Oesterl\'e, we 
completely characterize
those triples $(q,n,\ell)$ for which there are infinitely many solutions
$b$.  In all cases predicted by the $abc$ conjecture,
we are able (without any assumptions) to
prove there are indeed infinitely many solutions.
\end{abstract}

\section{Introduction}

Number theorists are often concerned with integer powers, with Fermat's 
``last theorem'' and Waring's problem being the two most prominent 
examples.
Another classic problem from number theory is the
{\it Nagell-Ljunggren problem}:  for which integers
$n, q \geq 2$ does the Diophantine equation
\begin{equation}
y^q = {{b^n - 1} \over {b-1}}
\label{nle}
\end{equation}
have positive integer solutions $(y,b)$?
See, for example, 
\cite{Nagell:1920,Nagell:1921,Ljunggren:1943a,Ljunggren:1943b,Oblath:1956,Shorey:1986,Le:1994,Hirata-Kohno&Shorey:1997,Bugeaud&Mignotte:1999a,Bugeaud&Mignotte:1999b,Bugeaud&Mignotte&Roy&Shorey:1999,Bugeaud&Mignotte&Roy:2000,Shorey:2000,Bennett:2001,Bugeaud&Hanrot&Mignotte:2002,Bugeaud:2002,Bugeaud&Mignotte:2002,Mihailescu:2007,Bugeaud&Mihailescu:2007,Mihailescu:2008,Browkin:2008,Kihel:2009,Laishram&Shorey:2012,Li&Li:2014,Bennett&Levin:2015}.

On the other hand, in combinatorics on words, repetitions of
strings play a large
role (e.g., \cite{Thue:1906,Thue:1912,Berstel:1995}).  If $w$ is a word (i.e., a
string or block of symbols chosen from a finite alphabet $\Sigma$), then
by $w \tothe n$ we mean the concatenation
$\overbrace{ww\cdots w}^n$.  (This is ordinarily written $w^n$, but we have
chosen a different notation to avoid any possible confusion with the power of
an integer.)  For example,
${\tt (mur)}\tothe 2 = {\tt murmur}$.

In this paper we combine both these definitions of powers and examine
the consequences.

In terms of the base-$b$ representation of both sides,
the Nagell-Ljunggren equation \eqref{nle} can be viewed as asking
when a power of an integer has base-$b$ representation
of the form $1\tothe n$ for some integer $n \geq 2$; such a number 
is sometimes called a ``repunit'' \cite{Yates:1978}.
An obvious generalization is to consider those powers of integers
with base-$b$ representation $a \tothe n$ for a single digit $a$;
such a number is somtimes called a ``repdigit'' \cite{Broughan:2012}.
This suggests an obvious further generalization of \eqref{nle}:
when does the power of an integer have a base-$b$
representation of the form $w\tothe n$ for some $n \geq 2$ and some
arbitrary word $w$ (of some given nonzero length $\ell$)?
In this paper we investigate this problem.

\begin{remark}
A related topic, which we do not examine here, is integer powers that
have base-$b$ representations that are palindromes.  See, for
example, \cite{Korec:1991,Hernandez&Luca:2006,Cilleruelo&Luca&Shparlinski:2009}.
\end{remark}

We introduce some notation.  Let $\Sigma_b = \{ 0,1,\ldots, b-1 \}$.
Let $b \geq 2$ be an
integer.   For an integer $n \geq 0$, we let
$(n)_b$ represent the canonical representation of $n$ in base $b$ (that is,
the one having no leading zeroes).  
For a word $w = a_1 a_2 \cdots a_n \in \Sigma_b^n$ we define
$[w]_b$ to be $\sum_{1 \leq i \leq n} a_i b^{n-i}$, the value of the
word $w$ interpreted as an integer in base $b$, and
we define $|w|$ to be the length of the word $w$
(number of alphabet symbols in it).    

Using this notation, we can express the class of equations we are
interested in:  they are of the form
\begin{equation}
 (y^q)_b = w\tothe n ,
\label{first}
\end{equation}
where $y, q, b, n \geq 2$ and $w \in \Sigma_b^*$.
Here we are thinking of $q$ and $n$ as given, and our goal is to
determine for which $b$ there exist
solutions $y$ and $w$.   Furthermore, we may classify
solutions $w$ according to their length $\ell = |w|$.

Alternatively, we can ask about the solutions to the equation
\begin{equation}
y^q = c {{b^{n\ell} - 1} \over {b^\ell - 1}} ,
\label{equiva}
\end{equation}
with $b^{\ell-1} \leq c < b^\ell$.  The correspondence of this equation with
Eq.~\eqref{first} is that $w = (c)_b$.
The inequality $b^{\ell-1} \leq c < b^\ell$ guarantees that 
the base-$b$ representation of $y^q$ is indeed an
$\ell$-digit string that does not start with the digit $0$.

Our results can be summarized as follows.  We call a triple of 
integers $(q,n, \ell)$ for $q, n \geq 2$ and $\ell \geq 1$
{\it admissible} if either
\begin{itemize} 
\item $(q,n) = (2,2)$,
\item $(n,\ell) = (2,1)$, or
\item $(q,n,\ell) \in \{ (2,3,1),(2,3,2),(3,2,2),(3,2,3),(3,3,1),(2,4,1),(4,2,2) \}$.
\end{itemize}
Otherwise $(q,n,\ell)$ is {\it inadmissible}.

Here is our main result:

\begin{theorem}
\leavevmode
\begin{itemize}
\item[(a)]
Assuming the $abc$ conjecture, there are only finitely many solutions
$(q,n,\ell,b,y,c)$ to \eqref{equiva} such that the triple $(q,n,\ell)$
is inadmissible.

\item[(b)]
For each admissible triple $(q,n,\ell)$, there are infinitely many solutions
$(b,y)$ to the equation $(y^q)_b = w \tothe n$ for $|w| = \ell$.
\end{itemize}
\label{main-thm}
\end{theorem}

In Section~\ref{abc-section} we prove (a) (as Theorem~\ref{abc}) and in
Section~\ref{admiss-section} we prove (b).  

One appealing distinction
between the Nagell-Ljunggren problem and the variant considered here is
that, for fixed $n$ and $q$, finding solutions to the classical
equation~\ref{nle} amounts to finding the integral points on a single
affine curve.  Provided that $(q,n) \notin \{ (2,2), (2,3), (3,2) \}$,
the genus of this curve is positive, so Siegel's theorem implies that
it has only finitely many integer points.  On the other hand, in the
variant considered here, for fixed $n$, $q$, and $\ell$, finding
solutions to Eq.~\eqref{equiva} amounts to finding integral points of controlled
height on a family of twists of a single curve, which is well known to
be a hard problem.  Moreover, there is an established literature of
using the $abc$ conjecture to attack such problems; for example, see
\cite{Granville:2007}.

We comment briefly on our representation of words.  In some cases,
particularly if $b \leq 10$, we write a word as a concatenation of digits.
For example, $1234$ is a word of length $4$.  However, if $b > 10$, this
becomes infeasible.  Therefore, for $b \geq 10$, we write
a word using parentheses and commas.  For example,
$(11,12,13,14)$ is a word of length $4$ representing
$40034$ in base $15$.

\section{Implications of the $abc$ conjecture}
\label{abc-section}

Let $\rad(n) = \prod_{p|n} p$ be the radical function,  the product
of distinct primes dividing $n$.  
We recall the $abc$ conjecture of Masser and Oesterl\'e
\cite{Masser:1985,Oesterle:1988}, as follows
(see, e.g., \cite{Stewart&Tijdeman:1986,Nitaj:1996,Browkin:2000,Granville&Tucker:2002,Robert&Stewart&Tenenbaum:2014}):

\begin{conjecture}
For all $\epsilon>0$, there exists a constant $C_\epsilon$ such that for all $a,b,c\in\Zee^+$ with $a+b=c$ and $\gcd(a,b)=1$, we have
$$c\leq C_\epsilon (\rad (abc))^{1+\epsilon}.$$
\end{conjecture}

We will need the following technical lemma. Its purpose will become clear in the proof of Theorem \ref{abc}. The proof is a straightforward manipulation of inequalities, but we include it for the sake of completeness.
\begin{lemma}\label{F positive}
Suppose that $q,n,\ell$ are positive integers with $q\geq 2$, $n\geq 2$, and $\ell\geq 1$. Further suppose that $(q,n,\ell)$ is not an admissible triple. Define
$$F(q,n,\ell)=\frac{24}{25}n\ell-1-\frac{n\ell}{q}-\ell.$$
Then $F(q,n,\ell)>0$.
\end{lemma}
\begin{proof}
If $(q,n,\ell)$ is not admissible, then either $n\geq 3$ or both $q\geq 3$ and $\ell\geq 2$. 

First assume that $n\geq 3$. Then
\[
F(q,n,\ell)=n\ell\left(\frac{24}{25} -\frac{1}{q}\right)-1-\ell \geq \frac{47}{25}\ell - \frac{3\ell}{q}-1.
\]
This quantity is positive if and only if
\[
q\geq \frac{75\ell}{47\ell-25}.
\]
For $\ell\geq 1$, the quantity $\frac{75\ell}{47\ell-25}$ is strictly less than $2$, and $q\geq 2$, so $F(q,n,\ell)>0$.

Now assume instead that $q\geq 3$ and $\ell\geq 2$. Rearranging the inequality in a different way, we see that $F(q,n,\ell)>0$ if and only if
\[
n\geq\frac{\ell q+q}{\frac{24}{25}\ell q -\ell}= 25\left(\frac{\ell+1}{\ell}\right)\left(\frac{q}{24 q-25}\right).
\]
This quantity is decreasing in $\ell$ and increasing in $q$ for all $\ell\geq 1$ and $q\geq 2$, and it is strictly less than $\frac{75}{48}$, which is less than 2. As $n\geq 2$, $F(q,n,\ell)$ is positive in this case as well.
\end{proof}

\begin{theorem}
Assume the $abc$ conjecture.
There are only finitely many solutions $(q,n,\ell,b,y,c)$ to the 
generalized Nagell-Ljunggren equation such that $(q,n,\ell)$ is an inadmissible triple.
\label{abc}
\end{theorem}

\begin{proof}
The equation $(y^q)_b=w\uparrow n$ can be written
\[
y^q = c\left(\frac{b^{n\ell}-1}{b^\ell-1}\right)
\]
for $c\in\Zee$ such that $(c)_b=w$. 
Note that $c\leq b^\ell-1$, so $y< b^{n\ell/q}$.

Suppose $p$ is a prime that divides $\frac{b^{n\ell}-1}{b^\ell-1}$. Then $p$ divides $y^q$, and thus $p^q$ divides $y^q$. Therefore
\[
y^q\geq \left(\rad\left(\frac{b^{n\ell}-1}{b^\ell-1}\right)\right)^q\geq
\frac{(\rad(b^{n\ell}-1))^q}{(\rad(b^\ell-1))^q},
\]
where we have used the obvious inequality $\rad(a/b)\geq \rad(a)/\rad(b)$. So
\[
\rad(b^{n\ell}-1)\leq y\rad(b^\ell-1)< yb^\ell<b^{n\ell/q+\ell},
\]
using $y< b^{n\ell/q}$ and $\rad(b^\ell-1)<b^\ell$.

Now consider the equation
\[
(b^{n\ell}-1)+1=b^{n\ell}.
\]
By the $abc$ conjecture, for all $\epsilon>0$, there is some positive constant $C_\epsilon$ such that
\[
b^{n\ell}\leq C_\epsilon (\rad((b^{n\ell}-1)(1)b^{n\ell}))^{1+\epsilon}\leq 
C_\epsilon(b\rad(b^{n\ell}-1))^{1+\epsilon},
\]
using that $b^{n\ell}$ and $b^{n\ell}-1$ are coprime and that $\rad(b^{n\ell})=\rad(b)\leq b$. We rewrite this inequality as
\[
b^{n\ell/(1+\epsilon)-1}\leq C_\epsilon^{1/(1+\epsilon)}\rad(b^{n\ell}-1).
\]
Set $C_\epsilon'=C_\epsilon^{1/(1+\epsilon)}$. 
Combining the upper and lower bounds on $\rad(b^{n\ell}-1)$,  we get
\[
b^{n\ell/(1+\epsilon)-1}\leq C_\epsilon' b^{n\ell/q+\ell}.
\]
Rearranging this, we have
\[
b^{n\ell /(1+\epsilon)-1-n\ell/q-\ell}\leq C_\epsilon',
\]
or equivalently,
\begin{equation}\label{fundamental inequality}
\frac{n\ell}{(1+\epsilon)}-1-\frac{n\ell}{q}-\ell \leq\frac{\log(C_\epsilon')}{\log(b)}.
\end{equation}
Recall that $y<b^{n\ell/q}$, or equivalently
\[
\frac{1}{\log(b)}<\frac{n\ell}{q\log(y)}.
\]
Therefore
\begin{equation}\label{inequality 2}
\frac{n\ell}{(1+\epsilon)}-1- \frac{n\ell}{q} -\ell \leq \log(C_\epsilon')\frac{n\ell}{q \log(y)}.
\end{equation}
In order for the triple $(q,n,\ell)$ to give rise to a solution of $(y^q)_b=w\uparrow n$, 
it is necessary that inequalities (\ref{fundamental inequality}) and (\ref{inequality 2}) are 
both satisfied. This puts restrictions on $b$ and $y$, respectively.

From this point forward, fix $\epsilon=\frac{1}{24}$. (Any fixed choice of $\epsilon<\frac{1}{23}$ would work for our purposes.) Let
\[
F(q,n,\ell)=\frac{n\ell}{(1+\epsilon)}-1-\frac{n\ell}{q}-\ell .
\]
It is easy to see that $F$ is increasing in $q$. We will soon see 
that $F$ is also increasing in $n$ and $\ell$ when $(q,n,\ell)$ is inadmissible.

It can be verified by an explicit calculation that $F(q,n,\ell)<0$ for all admissible triples $(q,n,\ell)$, including the infinite families with $(q,n)=(2,2)$ and $\ell$ arbitrary or $(n,\ell)=(2,1)$ and $q$ arbitrary. By Lemma \ref{F positive}, for every inadmissible triple $(q,n,\ell)$ we have $F(q,n,\ell)>0$, so there are only finitely many $b$ that satisfy inequality (\ref{fundamental inequality}). We will show that for large values of $n$ or $\ell$, no bases $b\geq 2$ satisfy (\ref{fundamental inequality}), and for large values of $q$, no $y\geq 2$ satisfy (\ref{inequality 2}) (clearly $y=1$ never gives a solution). Therefore, conditional on the $abc$ conjecture, there are only finitely many solutions to the generalized Nagell-Ljunggren equation that come from inadmissible parameters.

First we consider large values of $n$ or $\ell$ by computing lower bounds on the partial derivatives of $F$. 
Assume that $(q,n,\ell)$ is not admissible, 
and therefore either $n\geq 3$ and $q\geq 2$ or 
$n\geq 2$ and $q\geq 3$. Then we have lower bounds on the partial derivatives as follows:
\begin{align*}
\frac{\partial F}{\partial n} & = \ell\left(\frac{1}{1+\epsilon} - \frac{1}{q}\right)\geq 
1\left(\frac{1}{1+\epsilon} - \frac{1}{2}\right)= \frac{23}{50}\\
\frac{\partial F}{\partial \ell} & = n\left(\frac{1}{1+\epsilon}-\frac{1}{q}\right)-1\geq
\min\left(\frac{19}{75},\frac{19}{50}\right)=\frac{19}{75}
\end{align*}
If $n\geq 5$, then we have
\[
F(q,n,\ell)\geq F(2,n,1)\geq\frac{23}{50}(n-5)+F(2,5,1)>\frac{23}{50}(n-5),
\]
If $\ell\geq 5$, we have
\begin{align*}
F(q,n,\ell) & \geq \min(F(3,2,\ell),F(2,3,\ell))\\
& \geq\min\left((\ell-4)\frac{19}{75}+F(3,2,4),(\ell-5)\frac{19}{75}+F(2,3,3) \right) \\
& > \frac{19}{75}(\ell-4).
\end{align*}
Importantly, in the above calculations we have used both that $F(q,n,\ell)>0$ and that $F$ is increasing in $q$, $n$, and $\ell$ for all inadmissible triples. So $F(q,n,\ell)\to\infty$ as either $n\to\infty$ or $\ell\to\infty$. Thus for large values of either $n$ or $\ell$, 
inequality (\ref{fundamental inequality}) is not satisfied for any $b\geq 2$,
and there are no solutions to $(y^q)_b=w\uparrow n$.

It remains to show that large values of $q$ cannot be used in solutions. 
First we rewrite inequality (\ref{inequality 2}) as
\[
q\left(\frac{1}{1+\epsilon}-\frac{1}{n\ell}-\frac{1}{n}\right) \leq \frac{\log(C_\epsilon')}{\log(y)}+1.
\]
If $(q,n,\ell)$ is inadmissible, then either $n\geq 3$ and $\ell\geq 1$ or $n\geq 2$ and $\ell\geq 2$. So
\[
\frac{1}{n\ell}+\frac{1}{n} = \frac{1}{n}\left(1+\frac{1}{\ell}\right)\leq \frac{3}{4}
\]
and
\[
\frac{21q}{100}=q\left(\frac{24}{25}-\frac{3}{4}\right)\leq \frac{\log(C_\epsilon')}{\log(y)}+1\leq\frac{\log(C_\epsilon')}{\log(2)}+1,
\]
where we have replaced $y$ with $2$, which is the smallest value of $y$ that can be used in a solution. So for inadmissible triples $(q,n,\ell)$ with large values of $q$, inequality (\ref{inequality 2}) is not satisfied, and there are no solutions.

We have shown that there are only finitely many inadmissible triples that admit any solutions. By Lemma \ref{F positive} and inequality (\ref{fundamental inequality}), there are only finitely many bases $b$ that can appear in a solution corresponding to each such triple, and thus only finitely many solutions for each such triple. So the set of all inadmissible triples contributes in total only finitely many solutions.
\end{proof}
\begin{remark}
Shinichi Mochizuki, in a series of papers released in 2016, has
recently claimed a proof of the $abc$ conjecture.  If the proof is ultimately
verified, then Theorem~\ref{abc} will hold unconditionally.
\end{remark}

\section{Admissible triples}
\label{admiss-section}

In this section we examine each admissible triple and prove there are
infinitely many solutions.

\subsection{The case $(q,n) = (2,2)$}

\begin{theorem}
For each length $\ell \geq 1$, there are infinitely
many $b \geq 2$ such that the equation
$(y^2)_b = w \tothe 2$ has a solution with $|w| = \ell$.
\label{c22}
\end{theorem}

We need a lemma.

\begin{lemma}
For each integer $t \geq 0$ there exist infinitely
many integer pairs $(p,b)$ where $p \geq 2$ is prime and
$b \geq 2$ such that
$b^{2^t} \equiv \modd{-1} {p^2}$.  Furthermore, among
these pairs there are infinitely many distinct $b$.
\label{pbl}
\end{lemma}

\begin{proof}
By Dirichlet's theorem on primes in arithmetic
progressions, there are infinitely many primes
$p \equiv \modd{1} {2^{t+1}}$.  The group $G$ of integers
modulo $p^2$ is cyclic, and of order $p(p-1)$.  
Since $2^{t+1} \divides p-1$, there is an element
$b$ of order $2^{t+1}$ in $G$.  For this element
$b$ we have $b^{2^t} \equiv \modd{-1} {p^2}$.

To prove the last claim, note that for each fixed $t$
and fixed $b$ there are are only finitely many prime divisors
of $b^{2^t} + 1$.     If there were only finitely many distinct
$b$ among those pairs $(p,b)$ with
$b^{2^t} \equiv \modd{-1} {p^2}$, then there would only
be, in total, finitely many pairs $(p,b)$, contradicting
what we just proved.
\end{proof}

Now we can prove Theorem~\ref{c22}.

\begin{proof}
Let $\ell = r \cdot 2^t$, where $r$ is odd.
By Lemma~\ref{pbl} we know there exist infinitely
many $p$ and $b$ such that $b^{2^t} \equiv 
\modd{-1} {p^2}$.    Then
$b^\ell = b^{r \cdot 2^t} \equiv \modd{-1} {p^2}$.

Now write $b^\ell + 1 = m p^2$.  Then
$m^2 (b^\ell +1) = m^2 p^2$.    Choose
$v = \lceil {p \over {\sqrt{b}}} \rceil$.
Then
$$ {p \over {\sqrt{b}}} \leq v \leq {p \over {\sqrt{b}}} + 1,$$
so
$$ {{p^2} \over b} m \leq m v^2 \leq m \left({p \over {\sqrt{b}}} + 1 \right)^2 .$$
Hence
$$mv^2 \geq mp^2/b = {{{b^\ell} + 1} \over b} \geq b^{\ell - 1}.$$
Similarly, if $p \geq 5$, then
${p \over {\sqrt{2}}} + 1 \leq {p \over {1.1}},$
so 
$$ m v^2 \leq m \left({p \over {\sqrt{b}}} + 1 \right)^2
\leq m \left( {p \over {1.1}} \right)^2 \leq mp^2 - 1$$
if $p \geq 5$.

Then $(mvp)^2 = (mv^2) (b^\ell + 1)$.  The inequalities obtained
above imply that $mv^2$ in base $b$ is an $\ell$-digit number, so
the base-$b$ representation of
$(mvp)^2$ consists of two copies of $(mv^2)_b$, as desired.

From the second part of the Lemma, we get that there
are infinitely many $b$ corresponding to each length $\ell$.
\end{proof}

\begin{example}
Take $\ell = 12$.  Then $r = 3$ and $t = 2$.  If
$b = 110$ and $p = 17$, then $b^4 \equiv \modd{-1} {17^2}$.
Write $b^\ell + 1 = m \cdot p^2$, where $m = 10859613760280276816609$.
Let $v = \lceil {p \over {\sqrt{b}}} \rceil = 2$.
Then $mvp = 369226867849529411764706$ and
$((mvp)^2)_{b} = w\tothe 2$
where $$ w = [1, 57, 52, 15, 108, 52, 57, 94, 1, 57, 52, 16].$$
\end{example}

We now examine this case from a different angle, considering
$b$ to be fixed and examining for which pairs $(y,w)$ there
are solutions to $(y^2)_b = w\tothe 2$.

\begin{theorem}
For each base $b\geq 2$, the equation 
$$  (y^2)_b = w\tothe 2$$
has infinitely many solutions $(y,w)$.
\label{baseb}
\end{theorem}

First, we need a lemma:

\begin{lemma}
For all integers $b \geq 2$, there exists a prime $p \geq 5$ such
that $b$ has even order in the multiplicative group of
integers modulo $p^2$.
\label{p2}
\end{lemma}

\begin{proof}
First observe that if $b$ has even order mod $p$, then it must have
even order mod $p^2$.

If there is some prime $p\geq 5$ that divides $b+1$, then $p$ cannot
also divide $b-1$. Then $b^2\equiv 1\pmod{p}$ and $b\not\equiv
1\pmod{p}$, so $b$ has order 2 mod $p$, and we are done. Therefore it
suffices to prove the Lemma in the case that the primes dividing $b+1$
are a subset of $\{2,3\}$.

We aim to show that there is some prime $p\geq 5$ that divides $b^2+1$.
Then $p$ cannot also divide $b^2-1$, and so
\[b^4-1=(b^2+1)(b^2-1)\equiv 0\pmod{p}\] and $b$ has order 4 mod $p$.
Assume, to get a contradiction, that the only possible prime factors of
$b^2+1$ are $2$ and $3$.

By the Euclidean algorithm, $\gcd(b^2+1,b+1)= \gcd(1-b,b+1)=
\gcd(2,b+1)$, so $2$ is the only possible common prime divisor of both
$b^2+1$ and $b+1$. In particular, it is not possible that both numbers
are divisible by $3$.  Therefore one of $b+1$ or $b^2+1$ is a power of $2$;
thus $b$ is odd, and $\gcd(b+1,b^2+1)=2$. This leaves two possibilities:
either $b+1=2^n$ and $b^2+1=2 \cdot 3^m$, or $b+1=2 \cdot 3^m $
and $b^2+1=2^n$,
for some positive integers $n,m$.

If $b+1=2^n$, then $b+1\equiv 1$ or $2\pmod{3}$, so $b\equiv 0$ or
$1\pmod{3}$. But then $b^2+1$ cannot be divisible by 3. So instead we
must have $b+1=2 \cdot 3^m$, and
\[2^n=b^2+1=(2 \cdot 3^m-1)^2+1=2(2 \cdot 3^{2m}-2 \cdot 3^m+1).\] So $2^n$ is twice
an odd number, and $n=1$. But then $b=1$, which is a contradiction.
\end{proof}

We can now prove Theorem~\ref{baseb}:

\begin{proof}
Fix $b$, and
let $p \geq 5$ be a prime satisfying the conclusion of Lemma~\ref{p2}.
Let the order of $b$, modulo $p^2$, be $e' = 2e$ for some
integers $e', e \geq 1$.  

First, we claim that for all $b \geq 2$ such a $p$ can be chosen
such that there is an integer $t$
with
\begin{equation}
b^{-1/4} \sqrt{p} < t < \sqrt{9p/10} .
\label{ineq1}
\end{equation}

If $b \geq 16$, then the open interval
$(b^{-1/4} \sqrt{p}, \sqrt{9p/10})$ has length $> 1$ if $p \geq 5$, and
hence contains an integer.  

If $2 \leq b < 16$, we can use the $t$ and $p$ in the table below:
\begin{table}[H]
\begin{center}
\begin{tabular}{ccc}
$b$ & $p$ & $t$ \\
\hline
2& 5& 2 \\
3& 5& 2 \\
4& 5& 2 \\
5& 7& 2 \\
6& 7& 2 \\
7& 5& 2 \\
8& 5& 2 \\
9& 5& 2 \\
10& 7& 2\\
11&13& 3\\
12& 5& 2\\
13& 5& 2\\
14& 5& 2\\
15&13& 3\\
\end{tabular}
\end{center}
\end{table}

Hence, from \eqref{ineq1} we get 
$$ p^2/b <t^4 < .81 p^2$$
and so
$$ 1/b < t^4/p^2 < .81. $$

Now consider $z = (t^4/p^2)(b^{re} + 1)$ for odd $r \geq 1$.
Since $b$ has order $2e$ (mod $p^2$), we must have
$b^e \equiv \modd{-1} {p^2}$.   Then for odd $r \geq 1$ we have
$b^{re} \equiv \modd{-1} {p^2}$, and so
$z = {{t^4} \over {p^2}} (b^{re} + 1)$ is an integer.
From the previous paragraph we have
$$b^{re-1} < (t^4/p^2) b^{re} <
(t^4/p^2) (b^{re} + 1) = z,$$
and
$$ z = {{t^2}\over p} (b^{re}+1) < 0.81 (b^{re} + 1) < b^{re},$$
where the very last inequality holds provided $b^{re} \geq 5$.
If $b \geq 5$ this inequality holds for all $e$.  For smaller
$b$, we can choose $e$ as follows to ensure $b^{re} \geq 5$:
\begin{itemize}
\item if $b = 2$ then $p = 5$ and $e = 10$;
\item if $b = 3$ then $p = 5$ and $e = 10$;
\item if $b = 4$ then $p = 5$ and $e = 5$.
\end{itemize}

It follows that the base-$b$ representation of $z$ has
exactly $re$ digits.  Let $w = (z)_b$.  
Finally, note that
$$ [ww]_b = {{t^4} \over {p^2}} (b^{re} + 1) (b^{re} + 1) =
	({{t^2} \over p} (b^{re} + 1))^2 ,$$
so we can take $y = {{t^2} \over p} (b^{re} + 1)$.
\end{proof}

\begin{remark}
For $b = 2$ the solutions $y$ to the equation
$(y^2)_2 = w\tothe 2$ are given by the sequence
$$6, 820, 104391567, 119304648, 858993460, 900719925474100, \ldots,$$
which is sequence \seqnum{A271637} in the 
OEIS \cite{oeis}.
\end{remark}

\subsection{The case $(n,\ell) = (2,1)$}

This case, where $n = 2$ and $\ell = 1$, is the least interesting of
all the cases.  

\begin{proposition}
The equation $(y^q)_b = w \tothe 2$, $|w| = 1$, has infinitely
many solutions $b$ for each $q \geq 2$.
\end{proposition}

\begin{proof}
The equation can be rewritten as $y^q = c (b+1)$ for 
$1 \leq c < b$.  Given $q$, we can take $c = 1$, $y \geq 2$,
and $b = y^q - 1$.
\end{proof}

\subsection{The case $(q,n,\ell) = (2,3,1)$}

In this section we show 

\begin{theorem}
There are infinitely many bases $b$
for which the equation
$  (y^2)_b = w\tothe 3$ has a solution with $|w|= 1$.
\end{theorem}

\begin{proof}
We want to show there are infinitely many positive integer solutions to
\begin{equation}\label{2,3,1}
y^2 = c(b^2 + b + 1)
\end{equation}
with $1\leq c < b$. 
We show below that there are infinitely many 
integral points on the affine curve defined by 
\begin{equation}
3y^2 = x^2 + x + 1
\label{y3}
\end{equation}
with $x>0$.  Taking such a point, we easily obtain a solution to \eqref{2,3,1} with $c=3$, namely $(3y)^2 = 3(x^2+x+1).$

We rewrite \eqref{y3} as a norm equation in the real 
quadratic field $\mathbb{Q}(\sqrt{3})$.  In particular, rearranging terms yields
\[
(2x+1)^2 - 12y^2 = -3,
\] 
which is equivalent to $N((2x+1) + 2y\sqrt{3})=-3$, where $N$ is the 
norm from $\Que(\sqrt{3})$ to $\Que$.  Running this process in reverse, if $\alpha \in \mathbb{Q}(\sqrt{3})$ has norm $-3$ and can be written in the form $\alpha=a+b\sqrt{3}$ for positive integers $a,b$ with $a$ odd and $b$ even, then $x=(a-1)/2$, $y=b/2$ gives an integer point on $3y^2 = x^2+x+1$.

The unit group of $\mathbb{Z}[\sqrt{3}]$ (which is the ring of integers of $\Que(\sqrt{3})$) is 
generated by $-1$ and the fundamental unit $u=2-\sqrt{3}$, which has $N(u)=1$.  If $\alpha$ is any element of the desired form (e.g., $\alpha = 1 + 2\sqrt{3}$), then $\alpha u^{2k} = a_k + b_k\sqrt{3}$ will also have norm $-3$.  Moreover,   \[u^2=7-4\sqrt{3}\equiv 1\pmod{2\Zee[\sqrt{3}]},\] 
so that $\alpha u^{2k} \equiv \alpha \pmod{2\mathbb{Z}[\sqrt{3}]}$.  Thus, $a_k$ is odd and $b_k$ is even for every $k\in\mathbb{Z}$.  This gives infinitely many integer solutions to Eq.~\eqref{y3}, which, multiplying $a_k$ and $b_k$ by $-1$ if necessary, we may assume to have $x>0$.

\end{proof}
\begin{remark}
A similar class of solutions can be found for any $c>0$ for which the real quadratic field $\mathbb{Q}(\sqrt{c})$ has an integral element of norm $-3$.  Another such field is $\Que(\sqrt{7})$, for which the first associated solution is $(49^2)_{18} = [7,7,7]$.
\end{remark}

\subsection{The case $(q,n,\ell) = (2,3,2)$}

\begin{theorem}
There are infinitely many solutions to the case $(q,n,\ell) = (2,3,2)$.
\end{theorem}
\begin{proof}
We would like to find solutions to 
\begin{equation}\label{2,3,2}
y^2 = c(b^4+b^2+1)
\end{equation}
in positive integers $b,y,c$ such that $b\leq c < b^2$. Without loss of generality, any integer 
solution can be replaced by a positive integer solution. 

Notice that $x^4+x^2+1 = (x^2+x+1)(x^2-x+1)$, and suppose that $(x,y)$ is a integral point on the curve $3y^2=x^2+x+1$ 
such that $x^2-x+1$ is divisible by 49.  
Let $c=\frac{3}{49}(x^2-x+1)$, so that $3c(x^2-x+1)$ is a square. We compute
\[
\left(\sqrt{3c(x^2-x+1)}y\right)^2 = 3c(x^2-x+1)y^2 = c(x^2-x+1)(x^2+x+1) = c(x^4+x^2+1),
\]
which gives a solution to Eq.~\eqref{2,3,2} with $b=x$, as long as $x\leq c<x^2$; this inequality holds provided that $x\geq 18$.  We now produce infinitely many integral points on the curve $3y^2=x^2+x+1$ such that $49\mid x^2-x+1$, so this inequality is not an issue.  

As in the (2,3,1) case, we rewrite $3y^2 = x^2 + x + 1$ 
as the norm equation \[N((2x+1) + 2y\sqrt{3})=-3,\] where $N$ is the norm from $\Que(\sqrt{3})$ to $\Que$. 
Again as in the (2,3,1) case, if $\alpha\in\mathbb{Q}(\sqrt{3}$ has norm $-3$ and can be written as $\alpha=a+b\sqrt{3}$ for positive integers $a,b$ with $a$ odd and $b$ even, then $x=(a-1)/2$, $y=b/2$ gives 
an integer point on $3y^2 = x^2+x+1$. Observe that the polynomial $x^2-x+1$ will be divisible by $49$ if and only if either 
$x\equiv 19$ or $x\equiv 31\pmod{49}$. If $a=2x+1$, this occurs if and only if 
either $a\equiv 39$ or $a\equiv 63\pmod{98}$.

Observe that $\alpha=627 + 362\sqrt{3}$ satisfies the desired congruence 
conditions.  Let $u=2-\sqrt{3}$ be the fundamental unit of $\Zee[\sqrt{3}]$.  As $u$ is a unit, some power $u^r$ will necessarily be congruent to $1 \pmod{98 \mathbb{Z}[\sqrt{3}]}$; an explicit computation shows that $r=56$ works.  Thus, setting $\alpha u^{56k} = a_k + b_k \sqrt{3}$, for every $k\in \mathbb{Z}$ we have $a_k \equiv 39 \pmod{98}$ and that $b_k$ is even.  This produces the infinitely many solutions to Eq.~\eqref{2,3,2}.  The next one, with $k=1$, is
\begin{multline*}
(b,y,c)=(33519770429365238471302383574583401, \\
19352648480568478024495121554106701, \\
68790306712490710007811612444611710421528067927390557506093905927147).
\end{multline*}


\end{proof}

\subsection{The case $(q,n,\ell) = (3,2,2)$}

\begin{theorem}
There are infinitely many solutions to the case $(q,n,\ell)=(3,2,2)$.
\end{theorem}

\begin{proof}
The equation we want to solve in positive integers is
\begin{equation}\label{3,2,2}
y^3 = c(b^2+1)
\end{equation}
for $b^2\leq c < b^3$. 

We begin by showing that there are infinitely many integral points $(x,y)$ on the curve
\begin{equation}\label{3,2,2 curve}
2y^2 = x^2 + 1,
\end{equation}
where without loss of generality, both $x$ and $y$ are positive. Starting with such a point $(x,y)$ and 
rearranging Eq.~\eqref{3,2,2 curve}, we have $(2y)^3=4y(x^2+1)$, which gives a solution 
to Eq.~\eqref{3,2,2} with $b=x$ and $c=4y$. As $y=\sqrt{(x^2+1)/2}$, certainly $c\geq x$. The upper bound 
$c\leq x^2-1$ is equivalent to $2\sqrt{2(x^2+1)}\leq x^2-1$, which can be verified to hold for all $x\geq 4$, 
so all but finitely many of the integral points we find will produce solutions with $c$ in the correct range.

Eq.~\eqref{3,2,2 curve} is easily seen to be equivalent to the norm equation
\[
N(x+y\sqrt{2})=-1
\]
where $N$ is the norm from $\Que(\sqrt{2})$ to $\Que$. 
Let $u=1+\sqrt{2}$ be the fundamental unit of $\Que(\sqrt{2})$. Note that $N(u^k)=-1$ for 
all odd integers $k$. If we let $u^k=a_k+b_k\sqrt{2}$ for integers $a_k,b_k$, then the 
point $(a_k,b_k)$ is an integral point on Eq.~\eqref{3,2,2 curve}. So we have an infinite family of 
such points, and thus infinitely many solutions to Eq.~\eqref{3,2,2}.
\end{proof}

\subsection{The case $(q,n,\ell) = (3,3,1)$}

\begin{theorem}
There are infinitely many solutions to the case $(q,n,\ell) = (3,3,1)$.
\end{theorem}

\begin{proof}
We show that the equation 
\begin{equation}\label{(3,3,1) curve}
343y^2 = x^2+x+1
\end{equation}
has infinitely many solutions in integers with $x>y>0$. 
It follows that the equation 
\begin{equation}\label{3,3,1}
(7y)^3 =c(b^2+b+1)
\end{equation}
has infinitely many solutions with $c=y$, $b=x$, and $1\leq c<b$.

Completing the square on the RHS of Eq.~\eqref{(3,3,1) curve}, multiplying both sides by 4, and
rearranging, we obtain the equivalent equation 
\[
((2x+1)^2-(14y)^2(7))=-3.
\]
We write this as $N((2x+1) + 14y\sqrt{7}) = -3$, where $N$ is 
the norm from $\Que(\sqrt{7})$ to $\Que$. Let $\alpha=a+b\sqrt{7}$ with positive integers $a,b$. 
If $N(\alpha)=-3$, $a$ is odd, and $b$ is divisible by 14, then $x=(a-1)/2$ and $y=b/14$ yield 
a solution to Eq.~\eqref{(3,3,1) curve}. 

As in the previous theorems, we start with a single element with the desired properties, in this case $\alpha=37+98\sqrt{7}$, and use the unit group to produce infinitely many.  The fundamental unit 
of $\Que(\sqrt{7})$ is $u=8-3\sqrt{7}$, which satisfies $u^{14} \equiv 1 \pmod{14 \mathbb{Z}[\sqrt{7}]}$.  Thus, any of the elements $\alpha u^{14k}$ will be of the desired form, and there are infinitely many solutions to Eq.~\eqref{3,3,1} as well.

\end{proof}

\subsection{The case $(q,n,\ell) = (3,2,3)$}

The solutions we found to the $(3,3,1)$ case also produce solutions to
the $(3,2,3)$ case by a straightforward algebraic manipulation. Recall that for 
the $(3,2,3)$ case, the equation to solve is 
\begin{equation}\label{3,2,3}
y^3 = c(b^3+1)
\end{equation}
with $b^2\leq c<b^3$.

\begin{theorem}
There are infinitely many solutions to the case $(q,n,\ell) = (3,2,3)$.
\end{theorem}

\begin{proof}
Let $(y,b,c)$ be a solution to the case $(q,n,\ell) = (3,3,1)$.  Then
$y^3 = c(b^2 + b + 1)$ for some integer $c$ satisfying
$1 \leq c < b$. Set $b'=b+1$, $y'=y(b+2)$, and $c'=c(b+2)^2$. We claim 
that $(y',b',c')$ is a solution to Eq.~\eqref{3,2,3}. We compute
\begin{align*}
(y')^3 = y^3(b+2)^3 = c(b^2+b+1)(b+2)^3 = c(b+2)^2((b+1)^3+1) = c'((b')^3+1),
\end{align*}
as claimed. The only thing left to check is that $c'$ is in the correct range $(b')^2\leq c'< (b')^3$. 

As $1\leq c\leq b-1$ by assumption, we have 
\[
(b+2)^2\leq c(b+2)^2\leq (b-1)(b+2)^2.
\] 
So $c'\geq (b+2)^2\geq (b+1)^2=(b')^2$, 
and the lower bound on $c'$ is satisfied. For the upper bound, we directly compute 
\[
(b')^3 -(b-1)(b+2)^2= (b+1)^3 -(b-1)(b+2)^2 = 3b+5,
\]
and $3b+5 > 0$ for all $b$ under consideration. So 
$(b-1)(b+2)^2 < (b')^3$; thus $c'<(b')^3$, and $c'$ is in the correct range.

We have shown there are infinitely many solutions to the $(3,3,1)$ case, so 
it follows that there are infinitely many solutions to the $(3,2,3)$ case. Note that 
$w=(c,2c,c,2c,c,2c)$ in base $b'=b+1$.
\end{proof}

\subsection{The case $(q,n,\ell) = (2,4,1)$}

Here we will show that the equation
$$ (y^2)_b = w\tothe 4 $$
has solutions for infinitely many bases $b$.  

\begin{theorem}
There are infinitely many solutions to the case $(q,n,\ell)=(2,4,1)$.
\end{theorem}
\begin{proof}
The $(2,4,1)$ case requires solving the equation
\begin{equation}\label{2,4,1}
y^2 = c(b^3+b^2+b+1)
\end{equation}
for $1\leq c<b$. As in the (3,2,2) case, we use the
infinitely many integral points on the curve
\begin{equation}\label{2,4,1 curve}
2y^2 = x^2 + 1.
\end{equation}
Notice that any such $x$ is odd, and that we may assume that $x,y>0$ without loss of generality. 

Setting $b=x$ and $c=\frac{1}{2}(x+1)$, and multiplying both sides 
of Eq.~\eqref{2,4,1 curve} by $\frac{1}{2}(x+1)^2$, we obtain
\[
(y(x+1))^2 = \frac{1}{2}(x+1)(x+1)(x^2+1) = c(b^3+b^2+b+1),
\]
which gives a solution to Eq.~\eqref{2,4,1}. If $x>1$, we have $1\leq c < x$, as required.
\end{proof}

\subsection{The case $(q,n,\ell) = (4,2,2)$}

\begin{theorem}
There are infinitely many solutions to the $(q,n,\ell)=(4,2,2)$ case.
\end{theorem}
\begin{proof}
The key equation to solve for the $(4,2,2)$ case is
\begin{equation}\label{4,2,2}
y^4 = c(b^2+1)
\end{equation}
for $b\leq c<b^2$. We begin as in the $(3,2,2)$ case by finding infinitely many integral points 
on the curve
\begin{equation}\label{4,2,2 curve}
2y^2 = x^2+1,
\end{equation}
but now also insisting that $y$ be divisible by $13$. 
Assuming this is possible for the moment, set $b=x$ and 
$c= 2^3 \cdot 3^4 \cdot13^{-4} \cdot y^2=
2^3 \cdot 3^4 \cdot13^{-4} \cdot (x^2+1)$, and
note that $c$ is an integer. 
Clearly $c\leq x^2-1$; on the other hand, $c\geq x$ holds for all $x\geq 89$ 
and thus for all but finitely many of the integral solutions to Eq.~\eqref{4,2,2 curve}. 
Multiplying through by $(6y/13)^2$, we have
\[
2\left(\frac{6y}{13}\right)^4 = \frac{36}{169}y^2(x^2+1)=2c(x^2+1),
\]
so there are infinitely many solutions to Eq.~\eqref{4,2,2}.

It remains to demonstrate the existence of an infinite family of integral points $(x,y)$ on 
the curve Eq.~\eqref{4,2,2 curve} such that $13\mid y$. As in the $(3,2,2)$ case, the equation 
defining the curve can be rewritten $N(x+y\sqrt{2}) = -1$ where $N$ is the norm from 
$\Que(\sqrt{2})$ to $\Que$. Let $u=1+\sqrt{2}$ be the fundamental unit in $\Que(\sqrt{2})$. 
We compute $u^7 = 239+169\sqrt{2}$. Writing $u^{7k} = a_k + b_k\sqrt{2}$ for integers $a_k,b_k$, 
it is easy to see that $13\mid b_k$ for all $k\geq 1$, and $N(u^{7k})=-1$ for all odd $k$. So the family 
$u^{14k+7}$ gives an infinite supply of points of the desired form.
\end{proof}

We have now considered all admissible triples, and 
this completes the proof of Theorem~\ref{main-thm}.   \hfill $\qed$

\section{Solutions for inadmissible triples}

As we have seen above, the $abc$ conjecture implies that there are only
a finite number of solutions, in toto, corresponding to all inadmissible
triples, to the equation $(y^q)_b = w\tothe n$ with $|w| = \ell$.
So far we have found $8$ such solutions, and they are given
below in Table~\ref{inadmiss}.

\begin{table}[H]
\begin{center}
\begin{tabular}{ccclll}
$q$ & $n$ & $\ell$ & $b$ & $y$ & $w$  \\
\hline
2 & 5 & 1 & 3 & 11 & (1) \\
4 & 3 & 1 & 18 & 7 & (7) \\
4 & 2 & 3 & 19 & 70 & (9,13,4) \\
4 & 2 & 3 & 23 & 78 & (5,17,6) \\
5 & 2 & 2 & 239 & 52 & (27,203) \\
5 & 2 & 2 & 239 & 78 & (211,115) \\
6 & 2 & 2 & 239 & 26 & (22,150)  \\
3 & 2 & 4 & 12400 & 57459558593 & (4208, 7128, 8441, 5457) \\
\end{tabular}
\end{center}
\label{inadmiss}
\end{table}

We searched various ranges for other solutions and our search results
are summarized below.  

\def\bndd{486800}
\def\bnde{{10^7}}
\def\bndf{3764000}

\begin{center}
\scalebox{0.8}{
\begin{tabular}{cccc}
$q$ & $n$ & $\ell$ & $b$ \\
\hline
2 & 3 & 3 & none $\leq \bndf$ \\
2 & 3 & 4 & none $\leq \bndd$ \\
2 & 3 & 5 & none $\leq \bndd$ \\
2 & 4 & 2 & none $\leq \bndd$ \\
2 & 4 & 3 & none $\leq \bndd$ \\
2 & 4 & 4 & none $\leq \bndd$ \\
2 & 4 & 5 & none $\leq \bndd$ \\
2 & 5 & 1 & \red{one} $\leq \bnde$ \\
2 & 5 & 2 & none $\leq \bndd$ \\
2 & 5 & 3 & none $\leq \bndd$ \\
2 & 6 & 1 & none $\leq \bndd$ \\
2 & 6 & 2 & none $\leq \bndd$ \\
2 & 6 & 3 & none $\leq \bndd$ \\
3 & 2 & 4 & \red{one} $\leq \bnde$ \\
3 & 2 & 5 & none $\leq \bndd$ \\
3 & 2 & 6 & none $\leq \bndd$ \\
3 & 3 & 2 & none $\leq 5\cdot 10^5$ \\
3 & 3 & 3 & none $\leq \bndf$\\
3 & 3 & 4 & none $\leq \bndd$ \\
3 & 3 & 5 & none $\leq \bndd$ \\
3 & 4 & 1 & none $\leq \bndd$ \\
3 & 4 & 2 & none $\leq \bndd$ \\
3 & 4 & 3 & none $\leq \bndd$ \\
3 & 4 & 4 & none $\leq \bndd$ \\
3 & 4 & 5 & none $\leq \bndd$ \\
3 & 5 & 1 & none $\leq 5 \cdot 10^5$ \\
3 & 5 & 2 & none $\leq \bndd$ \\
3 & 5 & 3 & none $\leq \bndd$ \\
4 & 2 & 3 & \red{two} $\leq \bnde$ \\
4 & 2 & 4 & none $\leq 5 \cdot 10^5$ \\
\end{tabular}
\quad\quad\quad
\begin{tabular}{cccc}
$q$ & $n$ & $\ell$ & $b$ \\
\hline
4 & 2 & 5 & none $\leq \bndd$ \\
4 & 2 & 6 & none $\leq \bndd$ \\
4 & 3 & 1 & \red{one} $\leq \bnde$ \\
4 & 3 & 2 & none $\leq 5 \cdot 10^5$ \\
4 & 3 & 3 & none $\leq \bndf$\\
4 & 3 & 4 & none $\leq \bndd$ \\
4 & 3 & 5 & none $\leq \bndd$ \\
4 & 4 & 1 & none $\leq 5 \cdot 10^5$ \\
4 & 4 & 2 & none $\leq \bndd$ \\
4 & 4 & 3 & none $\leq \bndd$ \\
4 & 4 & 4 & none $\leq \bndd$ \\
4 & 4 & 5 & none $\leq \bndd$ \\
4 & 5 & 1 & none $\leq 5 \cdot 10^5$ \\
4 & 5 & 2 & none $\leq \bndd$ \\
4 & 5 & 3 & none $\leq \bndd$ \\
5 & 2 & 2 & \red{one} $\leq \bnde$ \\
5 & 2 & 3 & none $\leq 5 \cdot 10^5$ \\
5 & 3 & 1 & none $\leq 5 \cdot 10^5$ \\
5 & 3 & 2 & none $\leq 5 \cdot 10^5$ \\
5 & 3 & 3 & none $\leq \bndf$ \\
5 & 4 & 1 & none $\leq \bndd$ \\
5 & 4 & 2 & none $\leq \bndd$ \\
5 & 4 & 3 & none $\leq \bndd$ \\
6 & 2 & 2 & \red{one} $\leq \bnde$\\
6 & 2 & 3 & none $\leq 5 \cdot 10^5$ \\
6 & 3 & 1 & none $\leq \bndd$\\
6 & 3 & 2 & none $\leq \bndd$\\
6 & 3 & 3 & none $\leq \bndf$ \\
6 & 4 & 1 & none $\leq \bndd$ \\
\end{tabular}
}
\end{center}

\subsection{Our search procedure}

Consider Eq.~\eqref{equiva}:  $y^q = c {{b^{n\ell} - 1} \over {b^\ell - 1}}$.
We describe a search procedure to find solutions $(b,y)$ to this equation,
which produced the results above.
It has been
implemented in three different languages:  APL, Maple, and python.  Code
is available from the authors.

Given $(q,n, \ell)$ and $b$,
we start by factoring $ r := (b^{n\ell} - 1)/(b^\ell- 1)$.  This prime
factorization
can be speeded up using the algebraic factorization of the
polynomial $X^{(n-1)\ell} + \cdots + X^\ell + 1$ over $\Que[X]$.  For example,
if $n = 3$ and $\ell = 2$, the polynomial
$X^4 + X^2 + 1$ has the factorization $f(X) \cdot g(X)$ where
$f(X) = X^2 + X + 1$ and $g(X) = X^2 - X + 1$.
We therefore can compute $f(b)$ and $g(b)$ and factor each piece
independently and combine the results.

Now we have the prime factorization of $r$, say
$r = p_1^{e_1} \cdots p_t^{e_t}$,
If $cr$ is to be a $q$th power, then we must have that
$p_i^{q \lceil e_i/q \rceil}$ divides  $cr$ for $1 \leq i \leq t$.
So $c$ must be a multiple of
$$ d := \prod_{1 \leq i \leq t} p_i^{q \lceil e_i/q \rceil - e_i} .$$
and $c$ must further satisfy the inequality $b^{\ell - 1} \leq c < b^\ell$.
Writing $c = k^qd$ for some integer $k$, we have
$(b^\ell -1)/d \leq k^q <  b^\ell/d$ and so
$((b^\ell - 1)/d)^{1/q} \leq k < (b^\ell/d)^{1/q}$.  A solution then exists for
each integer $k$ in this interval, which can be easily checked.

The most time-consuming part of this calculation is the integer
factorization.  Typically we were searching some subrange of
the interval $[2, 10^7]$, with $n \leq 6$ and $\ell \leq 5$.
Thus we could be factoring numbers of size as large as
$10^{175}$.
If we want to perform this computation for many different
triples $(q,n,\ell)$ at once, it makes sense to first precompute the
algebraic factorizations described above, next compute
the factorizations of individual pieces, and finally assemble
the needed factorizations from these pieces.

\section{Beyond canonical base-$b$ representation}

One can consider the equation $(y^q) = w \tothe n$ for a wide variety of
other types of
representations.  In this section we consider two such other types of
representations.

\subsection{Bijective base-$b$ representation}

First, we consider the so-called ``bijective base-$b$ representation'';
see, for example, \cite{Foster:1947}; \cite[\S 9, pp.~34--36]{Smullyan:1961};
\cite[Solution to Exercise 4.1-24, p.\ 495]{Knuth:1969};
\cite[Note 9.1, pp.\ 90--91]{Salomaa:1973};
\cite[pp.\ 70--76]{Davis&Weyuker:1983};
\cite{Forslund:1995}; \cite{Boute:2000}.

This representation is like ordinary base-$b$ representation, except
that instead of using the digits $0,1, \ldots, b-1$, we use the
digits $1, 2, \ldots, b$ instead.  We use the notation
$\langle x \rangle_b$ to denote this representation.

\begin{theorem}
For all $b, \ell \geq 2$ there exists a word $w$ of length $\ell$
and an integer $y$ such that $\langle y^2 \rangle_b = w \tothe 2$.
\label{bij1}
\end{theorem}

\begin{proof}
Consider $y =b^\ell + 1$ for $\ell \geq 2$.
Then $y^2 =b^{2\ell} + 2b^\ell + 1$,
which has bijective base-$b$ representation $w\tothe 2$
for $w = ( (b-1) \tothe (\ell-2),  b, 1 )$.
\end{proof}

\def\fibbnd{435000}

For some specific bases $b$ there are other infinite families of
solutions.  For example, the table below summarizes
some of these families, for $n \geq 0$.  They are easy to prove by direct
calculation.  

\begin{table}[H]
\begin{center}
\begin{tabular}{ccc}
$b$ & $\langle y \rangle_b$ & $\langle y^2 \rangle_b$ \\
\hline
2 & $((12)\tothe {3n+3}) 212$ & $ (  ((221112)\tothe n) 221121112)\tothe 2 $ \\
3 & $((1331)\tothe {5n+2}) 22 $ & $ ((12132111223231233322)\tothe n) 1213211131)\tothe 2$ \\
4 & $((21)\tothe {5n+2}) 3$ & $ (( (1123421433)\tothe n) 11241)\tothe 2 $ \\
4 & $((24) \tothe {5n+2}) 4$ & $ (( (2143311234)\tothe n) 21434) \tothe 2 $ \\
5 & $((31) \tothe {3n+1}) 4 $ &$ (((155234)\tothe n) 211)\tothe 2$ \\ 
6 & $((41)\tothe {7n+3} )5$  & $ (((26211162534435)\tothe n) 2621121) \tothe 2$ \\
6 & $((46)\tothe {7n+3}) 6$ & $ (((42236551331456)\tothe n) 4223656) \tothe 2$ \\
7 & $ ( 3 \tothe {2n+2}) 4$ & $ (((15) \tothe {n+1}) 2)\tothe 2$ \\
8 & $ ((52) \tothe {n+1}) 6$ & $ (((34) \tothe {n+1}) 4)\tothe 2$ \\
9 & $  ((35) \tothe {10n+2}) 4$ & $ (((1385674932)\tothe {2n}) 13857)\tothe 2 $  \\
9 & $ ((53) \tothe {10n+2}) 6 $ & $ (((3213856749)\tothe {2n}) 32139)\tothe 2$ \\
9 & $ ((71) \tothe {10n+2}) 8 $ & $ (((5674932138)\tothe {2n}) 56751)\tothe 2$
\end{tabular}
\end{center}
\end{table}

\subsection{Fibonacci representation}

Yet another representation for integers involves the Fibonacci numbers.
The so-called Fibonacci or Zeckendorf representation of an integer $n \geq 0$
consists of writing $n$ as the sum of non-adjacent Fibonacci numbers:
$$ n = \sum_{2 \leq i \leq t} e_i F_i $$
where $e_i \in \{ 0, 1 \}$ and $e_i e_{i+1} \not= 1$ for $i \geq 2$;
see \cite{Lekkerkerker:1952,Zeckendorf:1972}.
In this case we write the representation of $n$, starting with the most
significant digit, as the binary word $(n)_F = e_t e_{t-1} \cdots e_2$.  

\begin{theorem}
There are infinitely many solutions to the equation
$(y^2)_F = w \tothe 2$, for integers $y$ and words $w$.
\end{theorem}

\begin{proof}
The proof depends on the following identity:
\begin{multline*}
(F_{4n+3} + F_{4n+6} + F_{8n+8} + F_{8n+11})^2 = \\
F_{4n+2} + F_{4n+5} + F_{4n+8} + F_{4n+10} + 
\left(\sum_{1 \leq i < n} F_{4n+4i+10} \right) + F_{8n+11} + \\
F_{12n+12} + F_{12n+15} + F_{12n+18} + F_{12n+20} + 
\left( \sum_{1 \leq i < n} F_{12n+4i+20} \right) + F_{16n+21} ,
\end{multline*}
which can be proved with a computer algebra system, such as Maple.

This identity shows that the Fibonacci representation of
$(F_{4n+3} + F_{4n+6} + F_{8n+8} + F_{8n+11})^2$
has the form $w^2$ with
$$ w = (1, 0,0,0,0, ((1, 0,0,0)\tothe (n-1)), 1,0,1,0,0,1,0,0,1, 
(0 \tothe (4n))) .$$
\end{proof}

\begin{remark}
There are other infinite families of solutions.  Here is one:
let $n \geq 1$ and 
suppose $(y)_F= ((100)\tothe (4n+2), 1,0,1,0,0,0)$.
Then $(y^2)_F = ww$ with
$w = ( ((1,0,0,1,(0\tothe 8))\tothe n), 1,0,0,1,0,0,0,0,0,0,1,0) $.
\end{remark}

A list of all the solutions to $(y^2)_F = w \tothe 2$
with $y < 34000000$ is given below.

\begin{table}[H]
\begin{center}
\begin{tabular}{rl}
$y$ & $w$ \\
\hline
4  &  100\\
49  &  10100100\\
306  &  100100000010\\
728  &  10000000101000\\
2021  &  1000100000101010\\
3556  &  10010101001000100\\
3740  &  10100101001000010\\
5236  &  100001010010010000\\
21360  &  100000010100101010010\\
35244  &  1000010000001010000000\\
98210  &  100100000000100100000010\\
243252  &  10000100010100100100000000\\
1096099  &  10010000010100100100010101010\\
1625040  &  100000010101001010000100001000\\
1662860  &  100001000100000010100000000000\\
4976785  &  10100010000100000000000010100100\\
5080514  &  10100100101001000000001000010100\\
11408968  & 1000010001000101001001000000000000\\
31622994  &  100100000000100100000000100100000010\\
31831002  &  100100000101010000010000010101000010\\
33587514  &  101000000100100010001000100101000000\\
33599070  &  101000000100101001000010000101000000\\
\end{tabular}
\end{center}
\end{table}

\begin{remark}
The only solutions to $(y^q)_F = w \tothe n$ other than $(q,n) = (2,2)$
that we found are $(2^4)_F = (100)\tothe 2$ and $(7^4)_F = (10100100)\tothe 2$.
\end{remark}

\appendix
\section{Tables of solutions for admissible triples}

Here we list some of the smaller solutions to Eq.~\eqref{first} corresponding
to admissible triples.

\begin{table}[H]
\begin{center}
\scalebox{0.8}{
\begin{tabular}{ccc}
$b$ & $y$ & $w$ \\
\hline
18& 49& (7)\\
22& 39& (3)\\
22& 78& (12)\\
30& 133& (19)\\
68& 247& (13)\\
68& 494& (52)\\
146& 1651& (127)\\
292& 4503& (237)\\
313& 543& (3)\\
313& 1086& (12)\\
313& 1629& (27)\\
313& 2172& (48)\\
313& 2715& (75)\\
313& 3258& (108)\\
313& 3801& (147)\\
313& 4344& (192)\\
313& 4887& (243)\\
313& 5430& (300)\\
423& 4171& (97)\\
423& 8342& (388)\\
439& 6231& (201)\\
499& 2289& (21)\\
499& 4578& (84)\\
499& 6867& (189)\\
499& 9156& (336)\\
\end{tabular}
}
\end{center}
\label{t231}
\caption{Solutions to $(q,n,\ell) = (2,3,1)$ with $b \leq 500$}
\end{table}

\begin{table}[H]
\begin{center}
\scalebox{0.8}{
\begin{tabular}{cccccc}
$b$ & $y$ & $w$ \\
\hline
68 & 160797 & (17, 53) \\
313 & 7575393 & (19, 32) \\
313 & 15150786 & (76, 128) \\
313 & 22726179 & (171, 288) \\
313 & 30301572 & (305, 199) \\
699 & 274088893 & (450, 133) \\
4366 & 4649437443 & (13, 2735) \\
4366 & 9298874886 & (54, 2208) \\
4366 & 13948312329 & (122, 2785) \\
4366 & 18597749772 & (218, 100) \\
4366 & 23247187215 & (340, 2885) \\
4366 & 27896624658 & (490, 2408) \\
4366 & 32546062101 & (667, 3035) \\
4366 & 37195499544 & (872, 400) \\
4366 & 41844936987 & (1103, 3235) \\
4366 & 46494374430 & (1362, 2808) \\
4366 & 51143811873 & (1648, 3485) \\
4366 & 55793249316 & (1962, 900) \\
4366 & 60442686759 & (2302, 3785) \\
4366 & 65092124202 & (2670, 3408) \\
4366 & 69741561645 & (3065, 4135) \\
4366 & 74390999088 & (3488, 1600) \\
4366 & 79040436531 & (3938, 169) \\
51567 & 134669813878873 & (49737, 9460) \\
234924 & 3266513519259697 & (14911, 203389) \\
234924 & 6533027038519394 & (59647, 108784) \\
234924 & 9799540557779091 & (134206, 186033) \\
686287 & 187720229347705587 & (231469, 166496) \\
3526450 & 27308133273274738941 & (1367399, 1295431) \\
3652434 & 40207131048785611981 & (2487105, 3465907) 
\end{tabular}
}
\end{center}
\caption{Solutions to $(q,n,\ell) = (2,3,2)$ with $b \leq 4 \cdot 10^6$}
\end{table}

\begin{table}[H]
\begin{center}
\scalebox{0.8}{
\begin{tabular}{ccc}
$b$ & $y$ & $w$ \\
\hline
7& 10 &  (2, 6) \\
38& 85 &  (11, 7)\\
41& 58 &  (2, 34)\\
41& 116 &  (22, 26)\\
57& 130 &  (11, 49)\\
68& 185 &  (20, 9)\\
117& 370 &  (31, 73)\\
239& 338 &  (2, 198)\\
239& 676 &  (22, 150)\\
239& 1014 &  (76, 88)\\
239& 1352 &  (181, 5)\\
268& 1105 &  (70, 25)\\
515& 2626 &  (132, 296)\\
682& 915 &  (2, 283)\\
682& 1220 &  (5, 494)\\
682& 1525 &  (11, 123)\\
682& 1830 &  (19, 218)\\
682& 2135 &  (30, 463)\\
682& 2440 &  (45, 542)\\
682& 2745 &  (65, 139)\\
682& 3050 &  (89, 302)\\
682& 3355 &  (119, 33)\\
682& 3660 &  (154, 380)\\
682& 3965 &  (196, 345)\\
682& 4270 &  (245, 294)\\
682& 4575 &  (301, 593)\\
682& 4880 &  (366, 244)\\
682& 5185 &  (439, 295)\\
682& 5490 &  (521, 430)\\
682& 5795 &  (613, 333)\\
882& 5365 &  (225, 55)\\
\end{tabular}
}
\end{center}
\caption{Solutions to $(q,n,\ell) = (3,2,2)$ with $b \leq 1000$}
\end{table}

\begin{table}[H]
\begin{center}
\scalebox{0.8}{
\begin{tabular}{ccc}
$b$ & $y$ & $w$  \\
\hline
18 & 7 &  (1) \\
18 & 14 & (8) \\
88916 & 24661 & (1897) \\
88916 & 49322 & (15176) \\
88916 & 73983 & (51219) \\
1147805 & 631111 & (190801) \\
6042955 & 3956043 & (1695447) \\
\end{tabular}
}
\end{center}
\caption{Solutions to $(q,n,\ell) = (3,3,1)$ with $b \leq 10^7$}
\end{table}

\begin{table}[H]
\begin{center}
\scalebox{0.8}{
\begin{tabular}{ccc}
$b$ & $y$ & $w$ \\
\hline
8 & 57 & (5, 5, 1)\\
19 & 140 & (1, 2, 1)\\
19 & 210 & (3, 14, 1)\\
19 & 280 & (8, 16, 8)\\
19 & 350 & (17, 5, 18)\\
23 & 234 & (1, 22, 18)\\
23 & 312 & (4, 16, 12)\\
23 & 390 & (9, 4, 22)\\
23 & 468 & (15, 21, 6)\\
31 & 532 & (5, 8, 1)\\
80 & 2709 & (6, 5, 29)\\
80 & 5418 & (48, 42, 72)\\
215 & 39438 & (133, 112, 42)\\
293 & 63042 & (116, 7, 101)\\
314 & 19005 & (2, 78, 41)\\
314 & 38010 & (17, 311, 14)\\
314 & 57015 & (60, 225, 165)\\
314 & 76020 & (143, 290, 112)\\
314 & 95025 & (281, 32, 101)\\
362 & 29337 & (4, 22, 117)\\
362 & 58674 & (32, 178, 212)\\
362 & 88011 & (109, 240, 263)\\
362 & 117348 & (259, 342, 248)\\
374 & 99645 & (135, 78, 189)\\
440 & 43617 & (5, 13, 393)\\
440 & 87234 & (40, 111, 64)\\
440 & 130851 & (135, 375, 51)\\
440 & 174468 & (322, 9, 72)\\
485 & 108342 & (47, 188, 433)\\
485 & 216684 & (379, 56, 69)\\
\end{tabular}
}
\end{center}
\caption{Solutions to $(q,n,\ell) = (3,2,3)$ with $b \leq 1000$}
\end{table}

\begin{table}[H]
\begin{center}
\scalebox{0.8}{
\begin{tabular}{ccc}
$b$ & $y$ & $w$ \\
\hline
7 &  20 &  (1)\\
7 &  40 &  (4)\\
41 &  1218 &  (21)\\
99 &  7540 &  (58)\\
239 &  20280 &  (30)\\
239 &  40560 &  (120)\\
1393 &  1373090 &  (697)\\
2943 &  4903600 &  (943)\\
8119 &  23308460 &  (1015)\\
8119 &  46616920 &  (4060)\\
45368 &  316540365 &  (1073)\\
45368 &  633080730 &  (4292)\\
45368 &  949621095 &  (9657)\\
45368 &  1266161460 &  (17168)\\
45368 &  1582701825 &  (26825)\\
45368 &  1899242190 &  (38628)\\
47321 &  527813814 &  (2629)\\
47321 &  1055627628 &  (10516)\\
47321 &  1583441442 &  (23661)\\
47321 &  2111255256 &  (42064)\\
82417 &  4091661910 &  (29905)\\
\end{tabular}
}
\end{center}
\caption{Solutions to $(q,n,\ell) = (2,4,1)$ with $b \leq 10^5$}
\end{table}

\begin{table}[H]
\begin{center}
\scalebox{0.8}{
\begin{tabular}{ccc}
$b$ & $y$ & $w$ \\
\hline
239 &   78 &  (2, 170) \\
239 &    104 &  (8, 136)\\
239 &    130 &  (20, 220)\\
239 &     156 &  (43, 91)\\
239 &    182 &  (80, 88)\\
239 &    208 &  (137, 25)\\
239 &    234 &  (219, 147)\\
682 &    305 &  (27, 191)\\
682 &    610 &  (436, 328)\\
4443 &    2810 &  (710, 3910)\\
12943 &    7930 &  (1823, 10935)\\
275807 &   78010 &  (1765, 45453)\\
275807 &   156020 &  (28242, 175634)\\
275807 &   234030 & (142978, 96202)\\
\end{tabular}
}
\end{center}
\caption{Solutions to $(q,n,\ell) = (4,2,2)$ with $b \leq 10^6$}
\end{table}


\begin{thebibliography}{10}

\bibitem{Bennett:2001}
M.~Bennett.
\newblock Rational approximation to algebraic numbers of small height: the
  {Diophantine} equation $|ax^n - by^n| = 1$.
\newblock {\em J. Reine Angew. Math.} {\bf 535} (2001), 1--49.

\bibitem{Bennett&Levin:2015}
M.~Bennett and A.~Levin.
\newblock The {Nagell-Ljunggren} equation via {Runge's} method.
\newblock {\em Monatsh. Math.} {\bf 177} (2015), 15--31.

\bibitem{Berstel:1995}
J.~Berstel.
\newblock {\em Axel {Thue's} Papers on Repetitions in Words: a Translation}.
\newblock Number~20 in Publications du Laboratoire de Combinatoire et
  d'Informatique {Math\'ematique}. Universit\'e du Qu\'ebec \`a Montr\'eal,
  February 1995.

\bibitem{Boute:2000}
R.~T. Boute.
\newblock Zeroless positional number representation and string ordering.
\newblock {\em Amer. Math. Monthly} {\bf 107} (2000), 437--444.

\bibitem{Broughan:2012}
K.~A. Broughan.
\newblock An explicit bound for aliquot cycles of repdigits.
\newblock {\em INTEGERS} {\bf 12} (2012), \#A15 (electronic).

\bibitem{Browkin:2000}
J.~Browkin.
\newblock The $abc$-conjecture.
\newblock In {\em Number Theory}, Trends Math., pp.  75--105. Birkh{\"a}user,
  2000.

\bibitem{Browkin:2008}
J.~Browkin.
\newblock A weak effective $abc$-conjecture.
\newblock {\em Funct. Approx. Comment. Math.} {\bf 39} (2008), 103--111.

\bibitem{Bugeaud:2002}
Y.~Bugeaud.
\newblock Linear forms in two $m$-adic logarithms and applications to
  diophantine problems.
\newblock {\em Compositio Math.} {\bf 132} (2002), 137--158.

\bibitem{Bugeaud&Hanrot&Mignotte:2002}
Y.~Bugeaud, G.~Hanrot, and M.~Mignotte.
\newblock Sur l'{\'e}quation diophantienne $(x^n - 1)/(x - 1) = y^q$. {III}.
\newblock {\em Proc. Lond. Math. Soc.} {\bf 84} (2002), 59--78.

\bibitem{Bugeaud&Mignotte:1999a}
Y.~Bugeaud and M.~Mignotte.
\newblock On integers with identical digits.
\newblock {\em Mathematika} {\bf 46} (1999), 411--417.

\bibitem{Bugeaud&Mignotte:1999b}
Y.~Bugeaud and M.~Mignotte.
\newblock Sur l'{\'e}quation diophantienne $(x^n-1)/(x-1)=y^q$. {II}.
\newblock {\em C. R. Acad. Sci. Paris} {\bf 328} (1999), 741--744.

\bibitem{Bugeaud&Mignotte:2002}
Y.~Bugeaud and M.~Mignotte.
\newblock L'{\'e}quation de {Nagell-Ljunggren} ${{x^n-1} \over {x-1}} = y^q$.
\newblock {\em Enseign. Math.} {\bf 48} (2002), 147--168.

\bibitem{Bugeaud&Mignotte&Roy:2000}
Y.~Bugeaud, M.~Mignotte, and Y.~Roy.
\newblock On the diophantine equation $(x^n-1)/(x-1)=y^q$.
\newblock {\em Pacific J. Math.} {\bf 193} (2000), 257--268.

\bibitem{Bugeaud&Mignotte&Roy&Shorey:1999}
Y.~Bugeaud, M.~Mignotte, Y.~Roy, and T.~N. Shorey.
\newblock The equation $(x^n-1)/(x-1)=y^q$ has no solution with $x$ square.
\newblock {\em Math. Proc. Cambridge Phil. Soc.} {\bf 127} (1999), 353--372.

\bibitem{Bugeaud&Mihailescu:2007}
Y.~Bugeaud and P.~Mih{\u{a}}ilescu.
\newblock On the {Nagell-Ljunggren} equation ${{x^n-1}\over {x-1}} = y^q$.
\newblock {\em Math. Scand.} {\bf 101} (2007), 177--183.

\bibitem{Cilleruelo&Luca&Shparlinski:2009}
J.~Cilleruelo, F.~Luca, and I.~E. Shparlinski.
\newblock Power values of palindromes.
\newblock {\em J. Combin. Number Theory} {\bf 1} (2009), 101--107.

\bibitem{Davis&Weyuker:1983}
M.~D. Davis and E.~J. Weyuker.
\newblock {\em Computability, Complexity, and Languages: Fundamentals of
  Theoretical Computer Science}.
\newblock Academic Press, 1983.

\bibitem{Forslund:1995}
R.~R. Forslund.
\newblock A logical alternative to the existing positional number system.
\newblock {\em Southwest J. Pure Appl. Math.} {\bf 1} (1995), 27--29.

\bibitem{Foster:1947}
J.~E. Foster.
\newblock A number system without a zero symbol.
\newblock {\em Math. Mag.} {\bf 21}(1) (1947), 39--41.

\bibitem{Granville:2007}
A.~Granville.
\newblock Rational and integral points on quadratic twists of a given
  hyperelliptic curve.
\newblock {\em Int. Math. Res. Not. IMRN}, No.\ 8, (2007), Art.~ID 027.

\bibitem{Granville&Tucker:2002}
A.~Granville and T.~J. Tucker.
\newblock It's as easy as $abc$.
\newblock {\em Notices Amer. Math. Soc.} {\bf 49} (2002), 1224--1231.

\bibitem{Hernandez&Luca:2006}
S.~Hern{\'a}ndez and F.~Luca.
\newblock Palindromic powers.
\newblock {\em Rev. Colombiana Mat.} {\bf 40} (2006), 81--86.

\bibitem{Hirata-Kohno&Shorey:1997}
N.~Hirata-Kohno and T.~N. Shorey.
\newblock On the equation $(x^m - 1)/(x - 1) = y^q$ with $x$ power.
\newblock In Y.~Motohashi, editor, {\em Analytic number theory}, Vol. 247 of
  {\em London Math. Soc. Lecture Note Series}, pp.  343--351. Cambridge
  University Press, 1997.

\bibitem{Kihel:2009}
O.~Kihel.
\newblock A note on the {Nagell-Ljunggren} {Diophantine} equation ${{x^n -1}
  \over {x-1}} = y^q$.
\newblock {\em JP J. Algebra Number Theory Appl.} {\bf 13} (2009), 131--135.

\bibitem{Knuth:1969}
D.~E. Knuth.
\newblock {\em The Art of Computer Programming. Volume 2: Seminumerical
  Algorithms}.
\newblock Addison-Wesley, 1969.
\newblock 1st edition.

\bibitem{Korec:1991}
I.~Korec.
\newblock Palindromic squares for various number system bases.
\newblock {\em Math. Slovaca} {\bf 41} (1991), 261--276.

\bibitem{Laishram&Shorey:2012}
S.~Laishram and T.~N. Shorey.
\newblock Baker's explicit $abc$-conjecture and applications.
\newblock {\em Acta Arith.} {\bf 155} (2012), 419--429.

\bibitem{Le:1994}
M.~H. Le.
\newblock A note on the diophantine equation ${{x^m -1}\over{x-1}} = y^n$.
\newblock {\em Math. Proc. Cambridge Phil. Soc.} {\bf 116} (1994), 385--389.

\bibitem{Lekkerkerker:1952}
C.~G. Lekkerkerker.
\newblock Voorstelling van natuurlijke getallen door een som van getallen van
  {Fibonacci}.
\newblock {\em Simon Stevin} {\bf 29} (1952), 190--195.

\bibitem{Li&Li:2014}
J.~Li and X.~Li.
\newblock On the {Nagell-Ljunggren} equation and {Edgar's} conjecture.
\newblock {\em Int. J. Appl. Math. Stat.} {\bf 52} (2014), 80--83.

\bibitem{Ljunggren:1943a}
W.~Ljunggren.
\newblock Einige {Bemerkungen} {\"u}ber die {Darstellung} ganzer {Zahlen} durch
  bin{\"a}re kubische {Formen} mit positiver {Diskriminante}.
\newblock {\em Acta Math.} {\bf 75} (1943), 1--21.

\bibitem{Ljunggren:1943b}
W.~Ljunggren.
\newblock Noen setninger om ubestemte likninger av formen $(x^n-1)/(x-1) =
  y^q$.
\newblock {\em Norsk. Mat. Tidsskr.} {\bf 25} (1943), 17--20.

\bibitem{Masser:1985}
D.~W. Masser.
\newblock Problem.
\newblock In {\em Symposium on Analytic Number Theory in Honour of K. F. Roth},
  p. ~25. Department of Mathematics, Imperial College, London, 1985.
\newblock Book of abstracts of conference.

\bibitem{Mihailescu:2007}
P.~Mih{\u{a}}ilescu.
\newblock New bounds and conditions for the equation of {Nagell-Ljunggren}.
\newblock {\em J. Number Theory} {\bf 124} (2007), 380--395.

\bibitem{Mihailescu:2008}
P.~Mih{\u{a}}ilescu.
\newblock Class number conditions for the diagonal case of the equation of
  {Nagell} and {Ljunggren}.
\newblock In {\em Diophantine approximation}, Vol.~16 of {\em Dev. Math.}, pp.
  245--273. Springer-Verlag, 2008.

\bibitem{Nagell:1920}
T.~Nagell.
\newblock Note sur l'{\'e}quation ind{\'e}termin{\'e}e $(x^n - 1)/(x - 1) =
  y^q$.
\newblock {\em Norsk. Mat. Tidsskr.} {\bf 2} (1920), 75--78.

\bibitem{Nagell:1921}
T.~Nagell.
\newblock Des {\'e}quations ind{\'e}termin{\'e}es $x^2 + x + 1 = y^n$ et $x^2 +
  x + 1 = 3y^n$.
\newblock {\em Norsk Mat. Forenings Skrifter}, 1921.

\bibitem{Nitaj:1996}
A.~Nitaj.
\newblock La conjecture $abc$.
\newblock {\em Enseign. Math.} {\bf 42} (1996), 3--24.

\bibitem{Oblath:1956}
R.~Obl{\'a}th.
\newblock Une propri{\'e}t{\'e} des puissances parfaites.
\newblock {\em Mathesis} {\bf 65} (1956), 356--364.

\bibitem{oeis}
N.~J. A.~Sloane et~al.
\newblock The on-line encyclopedia of integer sequences.
\newblock Available at \url{https://oeis.org}.

\bibitem{Oesterle:1988}
Joseph {Oesterl{\'e}}.
\newblock Nouvelles approches du ``th{\'e}or{\`e}me'' de {Fermat}.
\newblock {\em Ast{\'e}risque} {\bf 161-162} (1988), 165--186.
\newblock S{\'e}minaire Bourbaki, exp. 694.

\bibitem{Robert&Stewart&Tenenbaum:2014}
O.~Robert, C.~Stewart, and G.~Tenenbaum.
\newblock A refinement of the $abc$ conjecture.
\newblock {\em Bull. Lond. Math. Soc.} {\bf 46} (2014), 1156--1166.

\bibitem{Salomaa:1973}
A.~Salomaa.
\newblock {\em Formal Languages}.
\newblock Academic Press, 1973.

\bibitem{Shorey:1986}
T.~N. Shorey.
\newblock On the equation $z^q = (x^n - 1)/(x - 1)$.
\newblock {\em Indag. Math.} {\bf 48} (1986), 345--351.

\bibitem{Shorey:2000}
T.~N. Shorey.
\newblock Some conjectures in the theory of exponential diophantine equations.
\newblock {\em Publ. Math. (Debrecen)} {\bf 56} (2000), 631--641.

\bibitem{Smullyan:1961}
R.~M. Smullyan.
\newblock {\em Theory of Formal Systems}, Vol.~47 of {\em Annals of
  Mathematical Studies}.
\newblock Princeton University Press, 1961.

\bibitem{Stewart&Tijdeman:1986}
C.~L. Stewart and R.~Tijdeman.
\newblock On the {Oesterl{\'e}-Masser} conjecture.
\newblock {\em Monatsh. Math.} {\bf 102} (1986), 251--257.

\bibitem{Thue:1906}
A.~Thue.
\newblock {\"Uber} unendliche {Zeichenreihen}.
\newblock {\em Norske vid. Selsk. Skr. Mat. Nat. Kl.} {\bf 7} (1906), 1--22.
\newblock Reprinted in {\it Selected Mathematical Papers of Axel Thue}, T.
  Nagell, editor, Universitetsforlaget, Oslo, 1977, pp.~139--158.

\bibitem{Thue:1912}
A.~Thue.
\newblock {\"Uber} die gegenseitige {Lage} gleicher {Teile} gewisser
  {Zeichenreihen}.
\newblock {\em Norske vid. Selsk. Skr. Mat. Nat. Kl.} {\bf 1} (1912), 1--67.
\newblock Reprinted in {\it Selected Mathematical Papers of Axel Thue}, T.
  Nagell, editor, Universitetsforlaget, Oslo, 1977, pp.~413--478.

\bibitem{Yates:1978}
S.~Yates.
\newblock The mystique of repunits.
\newblock {\em Math. Mag.} {\bf 51} (1978), 22--28.

\bibitem{Zeckendorf:1972}
E.~Zeckendorf.
\newblock {Repr\'esentation} des nombres naturels par une somme de nombres de
  {Fibonacci} ou de nombres {Lucas}.
\newblock {\em Bull. Soc. Roy. {Li\'ege}} {\bf 41} (1972), 179--182.

\end{thebibliography}
\end{document}